\documentclass[12pt,reqno]{amsart}
\usepackage{amsmath,amsfonts,amssymb,amsthm}
\usepackage[scr]{rsfso}
\usepackage{Alegreya}

\def\R{\mathbb{R}} %real numbers
\def\C{\mathbb{C}} %complex numbers
\def\calE{\mathcal{E}}
\def\F{\mathcal{F}}
\def\manW{\mathcal{W}}
\def\manH{\mathcal{H}}
\def\manU{\mathcal{U}}
\def\FM{\F_M}
\def\FR{\F_\R}
\def\sysC{\mathscr{C}}
\def\sysF{\mathscr{F}} %try with rsfso
\def\I{\mathscr{I}}
\def\J{\mathscr{J}}
\def\V{\mathscr{V}}

\def\Vchat{\widehat{\V}}
\def\Vhat{\widehat{V}}

\def\ve{\mathsf{e}}
\def\vE{\mathsf{E}} % integral plane
\def\vf{\mathsf{f}} % R3-valued isometric immersion (not whether it'll look okay)
\def\vn{\mathsf{n}}
\def\vv{\mathsf{v}}
\def\vx{\mathsf{x}}

\def\pmap{\mathbf{p}}
\def\rmap{\mathbf{r}}

\def\w{\omega}
\def\etat{\tilde{\eta}}
\def\tpsi{\psi'}

\def\ri{\mathrm{i}}
\def\hatj{\hat{\jmath}}
\def\II{\operatorname{II}}
\def\di{\partial}

\def\sech{\operatorname{sech}}
\renewcommand{\Re}{\operatorname{Re}}
\renewcommand{\Im}{\operatorname{Im}}
\def\restr{\vert}
\def\intprod{\mathbin{\raisebox{.4ex}{\hbox{\vrule height .5pt width
5pt depth 0pt %
         \vrule height 3pt width .5pt depth 0pt}}}}
\newcommand{\hook}{\intprod}

\newtheorem{thm}{Theorem}
\newtheorem{prop}[thm]{Proposition}

\newtheorem{cor}[thm]{Corollary}
\newtheorem*{thm*}{Theorem}

\begin{document}
\title{Isometric Embedding for Surfaces: \\ Classical Approaches and Integrability}
\author{Thomas A. Ivey}
\address{Dept. of Mathematics, College of Charleston\\
66 George St., Charleston SC 29424}
\email{iveyt@cofc.edu}
\date{\today}

\begin{abstract}
The problem of finding a surface in 3-dimensional Euclidean space with a given first fundamental form was classically known as Bour's problem, and is closely related to the problem of determining which surfaces in space are `applicable' (i.e., isometric without being congruent) to a given surface.  We review classical approaches to these problems, including coordinate-based approaches exposited by Darboux and Eisenhart, as well as the moving-frames based approaches advocated by Cartan.  In particular, the first approach involves reducing Bour's problem to solving a single PDE; settling the question of when this PDE is integrable by the method of Darboux is surprisingly easy, once we answer the analogous question for the isometric embedding system arising from the moving frames approach.  The latter question was settled in recent joint work with Clelland, Tehseen and Vassiliou.\end{abstract}
\maketitle

\section{Introduction}\label{introsec}
The problem of finding a surface in 3-dimensional Euclidean space with a given metric was classical known as {\em Bour's problem}.
The most notable advance in this 19th century in this area is sometimes known as the {\em fundamental theorem of surface theory}, and is due to Bonnet.  We state it here in modern language (cf. \cite{Spivak}, Vol. 3, Ch. 2).

\begin{thm*}[Bonnet, 1867] Let $(M,g)$ be an abstract Riemannian surface, and let $h$ be a symmetric 2-tensor on $M$ satisfying the conditions
\begin{align} \det( g^{-1} h ) &=  K,  \label{GaussEq} \\
		\nabla_i h_{jk} - \nabla_j h_{ik} &= 0 \label{CodazziEq},
\end{align}
where $K$ is the Gauss curvature of $g$ and $\nabla_i h_{jk}$ denotes the components in local coordinates of the covariant derivative of $h$ with respect to the Levi-Civita connection of $g$.  Then given any point $p\in M$, there exists an open set $\manU\subset M$ containing $p$, and an isometric embedding $\vf:\manU\to\R^3$
with second fundamental form $h$.  Moreover, $\vf$ is unique up to rigid motion.
\end{thm*}

Of course, \eqref{GaussEq} and \eqref{CodazziEq} are respectively known as the Gauss and Codazzi equations, and the latter can be more succinctly formulated as saying $\nabla h$ is a totally symmetric 3-tensor.  In terms of local coordinates $x^1, x^2$ on $M$, and the components of $g$ in these coordinates, the equations we are trying to solve are
\begin{equation}\label{isodot}
\dfrac{\di \vf}{\di x^i} \cdot \dfrac{\di \vf}{\di x^j} = g_{ij}.
\end{equation}
Below, we will reinterpret Bonnet's result as saying that, when the isometric immersion system \eqref{isodot} is reformulated in terms of moving frames, this results in an overdetermined set of linear differential equations for which the Gauss-Codazzi system is the solvability condition.

In hope of writing down solutions of \eqref{isodot} in at least some particular cases, 19th-century geometers such as Darboux and Weingarten sought to reduce the solvability conditions to a single scalar PDE.  Two approaches to this are described in \cite{Darboux}; we discuss the first here, following the exposition in \cite{Han}, Chapter 3.  (The second approach is the method of Weingarten, which we hope to discuss in a future paper.)

Let $\vf:M\to \R^3$ be an isometric immersion and let $\vn:M \to \R^3$ be a unit normal vector field along the image.  Regarding $\vf$ and $\vn$ as vector-valued functions on $M$, differentiating $\vn \cdot d\vf = 0$ shows that the second fundamental form of the image surface is given by
\begin{equation}\label{hbynabla}
h = \vn \cdot \nabla^2 \vf = -d\vn \cdot d\vf.
\end{equation}
Fix a unit vector $\hatj$ in $\R^3$, and on $M$ define the function $u = \hatj \cdot \vf$.   Differentiating this gives
\begin{equation}\label{doon}
du = \hatj \cdot d\vf
\end{equation}
and $\nabla^2 u = \hatj \cdot \nabla^2 \vf$, and then from \eqref{hbynabla} we have
$$\nabla^2 u = (\hatj \cdot \vn) h.$$
at points where $\hatj \cdot \vn \ne 0$.  This in turn, along with the Gauss equation \eqref{GaussEq}, implies that
$$\det( g^{-1} \nabla^2 u ) = (\hatj \cdot \vn)^2 K.$$
Moreover, resolving $\hatj$ in terms of its tangential and normal components and using \eqref{doon} shows that $(\hatj \cdot \vn)^2 = 1 - |du|^2$,
where the norm on the right is taken with respect to $g$.  As a result, we see that each component of $\vf$ must satisfy the scalar PDE
\begin{equation}\label{DarbouxPDE}
\det( g^{-1} \nabla^2 u ) = (1-|du|^2) K,
\end{equation}
which has come to be known as the {\em Darboux equation}.  Conversely, given a single solution $u$ of \eqref{DarbouxPDE} such that $|du|^2 < 1$, the metric $g-du^2$ is nondegenerate and flat, and thus there exist local coordinates $x,y$ such that $g - du^2 = dx^2 + dy^2$ (see, e.g., Lemma 3.1.1 in \cite{Han}), and so $\vf=(x,y,u)$ constitutes a (local) isometric immersion .

The main object of the present note is determining the set of metrics for which the partial differential
equation \eqref{DarbouxPDE} is integrable in the sense of the {\em method of Darboux}.
The article is organized as follows.  In \S2 we review some basic notions from moving frames and the theory of exterior differential systems (EDS), including the concepts of prolongation and integrable extension, and the definition of Darboux integrability.  In \S3 we formulate the isometric immersion problem as an EDS $\I_0$, and show how it is related to the Gauss-Codazzi system and to the Darboux equation.  In \S4 we show that the set of metrics for which \eqref{DarbouxPDE} is Darboux-integrable at the 2-jet level coincides exactly with the set of metrics for which $\I_0$ is Darboux-integrable; the latter classification was obtained in \cite{CITV}.

\section{Exterior Differential Systems}
In this section we review essential definitions used in later sections; more detailed discussion of these concepts may be found in \cite{CfB2}.
%We begin with the concept of Darboux integrability, and then will

Let $M^m$ be a smooth manifold, let $\Omega^k (M)$ denote the $C^\infty(M)$-module of smooth differential $k$-forms on $M$, and let
$\Omega^\bullet(M) = \oplus_{k=0}^m \Omega^k(M)$.
This is an algebra with respect to the wedge product, and an {\em exterior differential system} on $M$
 is a homogeneous ideal in this algebra which is also closed under exterior differentiation.
 (For the sake of simplicity, we'll assume the ideal contains no 0-forms.)
An EDS $\I$ may be specified by a giving a list of differential forms on $M$ which are algebraic generators for the ideal, 
%in which case we will write $\I = \langle \ldots \rangle_\alg$ with the generators between the angle brackets, 
or more briefly giving a list of differential forms which generate the ideal along with their exterior derivatives.
%, in which case we will write $\I = \langle \ldots \rangle_\diff$.  
Ideals generated by 1-forms and their derivatives are known as {\em Pfaffian systems}.  For such systems, we will let $I \subset T^*M$ denote the vector sub-bundle spanned by the 1-forms of the ideal; clearly, specifying the bundle $I$ is equivalent to specifying the Pfaffian system.
If $r$ is the rank of this vector bundle then we say $r$ is also the rank of the Pfaffian system generated by sections of $I$ and their exterior derivatives.  If the exterior derivatives of sections of $I$ are already in the algebraic ideal generated by those sections, the Pfaffian system is said to be Frobenius and (at least locally) can be generated by $r$ closed 1-forms.

The analogue of a `solution' for an EDS $\I$ is a submanifold $S\subset M$ such that every differential form in $\I$ pulls back, under the inclusion map,  to vanish on $S$.  These are called {\em integral submanifolds} of $\I$.  Exterior differential systems are often accompanied by non-degeneracy
conditions which take the form of an {\em independence condition} defined by a decomposable $n$-form $\Psi \in \Omega^n (M)$.  In that case, an $n$-dimensional integral submanifold $S$ is said to be {\em admissible} if $\Psi$ pulls back to be nowhere vanishing on $S$.  (The form $\Psi$ need only be well-defined up to nonzero multiple, and the addition of $n$-forms in $\I$)

To define the prolongation of $\I$, we need to consider its set of integral elements.
An $n$-dimensional subspace $\vE \subset T_p M$ is said to be an {\em integral element} of EDS $\I$ if $\vv\hook \varphi = 0$ for all vectors $\vv \in \vE$ and all $\varphi \in \I$.  We define $\calE_n(\I) \subset G_n(TM)$ as the set of all $n$-dimensional integral elements of $\I$,
and when $\I$ has an independence condition we let $\calE_n(\I,\Psi)$ denote the set of integral $n$-planes $\vE$ satisfying the independence condition (i.e.,
$\Psi\restr_\vE \ne 0$).  In most cases, $\calE_n(\I,\Psi)$ forms a smooth fiber bundle over $M$, so we will assume this.

The Grassmann bundle $G=G_n(TM)$ carries a natural Pfaffian system $\sysC$, known as the canonical contact system.  The 1-forms of $\sysC$
are sections of a vector bundle $C \subset T^*G$ whose fiber at a point $(p,\vE) \in G$
is $\pi^* \vE^\perp$, where $\pi: G\to M$ is the bundle map and $\vE^\perp$ denotes the annihilator of $\vE$ inside $T_p^* M$.
Then the {\em prolongation} of $\I$ is the pullback of $\sysC$ to $\calE_n(\I,\Psi)$, with independence condition given by $\pi^* \Psi$.  In practice, we calculate the prolongation by introducing fiber coordinates on $\calE_n(\I,\Psi)$ and writing the 1-forms that annihilate an admissible integral $n$-plane in terms of these coordinates.

In subsequent sections, we will re-formulate the problem of isometrically embedding a surface into Euclidean space as the problem of constructing admissible integral surfaces for an exterior differential system, and similarly recast the Gauss-Codazzi system and the Darboux equation as equivalent Pfaffian systems.  In particular, we will see that the Darboux equation is essentially the solvability condition for the Gauss-Codazzi system; this relationship will be made precise through the notion of {\em integrable extension}.

\subsection*{Integrable Extensions}
% do integrable extensions
Let $\I$ be an EDS on manifold $M$,  let $\pi:N \to M$ be a submersion.  Then the pullbacks to $N$ of all differential forms of $\I$ 
form an EDS on $N$ which we denote by $\pi^* \I$.  But it is sometimes also possible to add to this ideal some new 1-forms which are not basic for $\pi$, and which form a new EDS without adding any generators of higher degree.
Specifically, an EDS $\J$ on $N$ is said to be an {\em integrable extension} of $\I$ if
$\pi^* \I \subset \J$ and there exists a sub-bundle $E \subset T^*N$ whose sections are transverse to the fibers of $\pi$
(i.e., the pairing between 1-forms in $E$ and vector fields tangent to the fiber is nondegenerate)
and such that $\J$ is algebraically generated by sections of $E$ and the differential forms in $\pi^* \I$.
It follows that if $S \subset M$ is an integral submanifold of $\I$, then $\J$ pulls back to $\pi^{-1}(S)$ to be a Frobenius system.

\subsection*{Darboux Integrability for Monge Systems}
We define a {\em Monge exterior differential system of class $k$} as an EDS defined on a manifold of dimension $k+4$ and which,
in the vicinity of any point, is algebraically generated by 1-forms $\psi_1, \ldots, \psi_k$ and two 2-forms $\Omega_1, \Omega_2$
(both sets assumed to be pointwise linearly independent).  For example, any second order PDE for one function of two variables is equivalent to
a Monge system of class 3, and if the PDE is a Monge-Amp\`ere equation then it is equivalent to a Monge system of class 1 (see, e.g., Chapter 7 in 
\cite{CfB2}).  Thus, Monge systems of class 1 also called Monge-Amp\`ere exterior differential systems.
%(This will be illustrated for the Darboux equation in \S\ref{DEsec} below.)
Mirroring the typology of the corresponding PDE systems, Monge systems are divided into types as follows.

Again letting $I \subset T^*M$ denote the bundle spanned by the generator 1-forms of a class $k$ Monge system $\I$, we let $W \subset \Lambda^2(T^*M / I)$ be the bundle spanned by the generator 2-forms.
Then for $\Upsilon \in C^\infty(W)$ we define $$Q(\Upsilon) = \Upsilon \wedge \Upsilon \wedge \psi_1 \ldots \psi_k.$$
This is only well-defined up to multiple at each point of $M$, as we could choose a different set of 1-form generators.
Using the identification of the fiber of $\Lambda^{k+4} T^*M$ with $\R$, $Q$ becomes a scalar-valued quadratic form on the fiber of $W$ which is well-defined up to a nonzero multiple.
We say the Monge system is {\em elliptic, hyperbolic, or parabolic} at $p \in M$ according to whether this quadratic form is respectively nondegenerate and definite, nondegenerate and indefinite, or degenerate on the fiber $W_p$.  We will limit our attention to Monge systems where the type is constant.

For hyperbolic Monge systems we let $\Upsilon_+$ and $\Upsilon_-$ denote local sections of $W$ which are null directions for $Q$.
These 2-forms are decomposable, modulo sections of $I$, and we define the associated {\em characteristic Pfaffian systems} $\V_+$ (resp. $\V_-$)
as generated by the factors of $\Upsilon_+$ (resp. $\Upsilon_-$) together with $\psi_1, \ldots \psi_k$.
A hyperbolic Monge system is defined to be {\em Darboux-integrable} if each characteristic system contains a pair of closed 1-forms
which are linearly independent from each other and the sections of $I$.
Such systems were classically known to be solvable by ODE methods (see, e.g., Chapter IX in \cite{gursa}),
but the understanding of these systems, and their solvability by means of group actions,
has been greatly advanced by the recent work of Anderson, Fels and Vassiliou \cite{AFB,AFV}.  (In those papers, the characteristic systems are called {\em singular systems.})

For an elliptic Monge system $\I$, there are complex conjugate local sections $\Upsilon, \overline{\Upsilon}$ of $W \otimes \C$ which are null directions for the $\C$-linear extension of $Q$.  We analogously define the characteristic systems $\V, \overline{\V}$ to be generated by the sub-bundles of
$T^*M \otimes \C$ spanned by $\psi_1, \ldots, \psi_k$ and the factors of $\Upsilon$ (resp. $\overline{\Upsilon}$), and say that
$\I$ is Darboux-integrable if $\V$ contains a pair linearly independent closed (complex-valued) 1-forms (which, by complex conjugation, implies the same for $\overline{\V}$).

For parabolic Monge systems, Darboux integrability is not applicable, since there is only one characteristic system and the existence of first integrals for that system does not enable us to reduce to solving ODE systems.

We also note that Darboux integrability is relatively easy to check. Recall that, for a Pfaffian system
generated by sections of $I\subset T^*M$, the derived system of a $\I$ is generated by the sections of $I$ that are in the kernel of the map
$\delta_I: \theta \mapsto d\theta \mod I$ -- in other words, by 1-forms whose exterior derivatives are in the algebraic ideal generated purely
by the 1-forms of $\I$.  Applying this construction iteratively results in a sequence of vector bundles $I \supset I^{(1)} \supset I^{(2)} \ldots $,
where each member of which is the derived system of the previous one, and may be calculated using only differentiation and linear algebra.
Since this {\em derived flag} terminates in the highest-rank Frobenius system contained in $\I$, the Darboux integrability condition requires that the derived flag of each characteristic system terminate in a Frobenius system of rank at least two.

The prolongation of a Monge system of class $k$ is a Monge system of class $k+2$ of the same type, and while the initial system may not be Darboux-integrable, it may turn out that its prolongation is.  For example, we will see in \S\ref{DEsec} that the
PDE \eqref{DarbouxPDE} is equivalent to a Monge-Amp\`ere system $\J_0$ that, for certain metrics, becomes integrable after one prolongation.  Since this system $\J_0$
lives on the 5-dimensional space $J^1(M,\R)$ of 1-jets of functions on $M$, whereas its prolongation lives on the space of 2-jets, we say that
\eqref{DarbouxPDE} becomes Darboux-integrable at the 2-jet level.  There are examples given in Ch. VII of \cite{gursa} of second-order PDE that become Darboux-integrable at the $n$-jet level for arbitrarily $n$.

%The bundle $V$ spanned by the 1-forms of a given characteristic system $\V$
%sits at the head of a flag $V \supset V^{(1)} \supset V^{(2)} \ldots $ of vector bundles.  Each member of the flag is the derived system of the previous one,
%and may be calculated using only differentiation and linear algebra; see Chapter 7 in \cite{CfB2} for details.
%The derived flag of $V$ terminates in the bundle $V^{(\infty)}$ spanned
%by the 1-forms of the highest-rank Frobenius system contained in $\V$; thus $\I$ is Darboux-integrable only if the derived flag
%of each characteristic system terminates in a system of rank at least two.

\section{Isometric Immersions via Moving Frames}
For an oriented Riemannian surface $M$, we will let $\FM$ denote the oriented orthonormal frame bundle of $M$.
This is a principal $SO(2)$-bundle and carries canonical 1-forms $\eta^1, \eta^2$ which are semibasic for the projection to $M$.  If
$\sigma$ is a local section of $\FM$, defined by an oriented orthonormal frame $\vv_1, \vv_2$, then $\etat^a =\sigma^* \eta^a$
is the dual coframe, where $1\le a \le 2$.  The 1-forms $\eta^1, \eta^2$, together with the connection 1-form $\eta^1_2$, comprise a global coframing on $\FM$, and these 1-forms satisfy the
structure equations
\begin{equation}\label{Mstruc}
d\eta^1 = -\eta^1_2 \wedge \eta^2, \quad d\eta^2 = -\eta^2_1 \wedge \eta^1, \quad d\eta^1_2 = K\eta^1 \wedge \eta^2,
\end{equation}
where $\eta^2_1=-\eta^1_2$, and $K$ is the Gauss curvature which is a well-defined function on $M$.

On the Euclidean side, we let $\FR$ denote the oriented orthonormal frame bundle of $\R^3$, which is a principal $SO(3)$-bundle.
Here, the canonical 1-forms $\w^i$ and connection forms $\w^i_j=-\w^j_i$ for $1\le i,j \le 3$ are readily defined
as coefficients in the derivatives of the vector-valued basepoint function $\vx$ and functions $\ve_1, \ve_2, \ve_3$ giving the members of the frame:
\begin{align}\label{d-of-basepoint} d\vx &= \ve_i \w^i, \\
\label{d-of-frame} d\ve_i &= \ve_j \w^j_i,
\end{align}
where we adopt the usual summation convention from now on.
Differentiating these equations leads in turn to the structure equations
$$d\w^i = -\w^i_j \wedge \w^j, \qquad d\w^i_j = -\w^i_k \wedge \w^k_j.$$

An isometric immersion $\vf:M \to \R^3$ induces a lift $\vf': \FM \to \FR$ that maps anypair of oriented orthonormal vectors $\vv_1, \vv_2$ at $p \in M$
to the orthonormal frame at $\vf(p)$ given by $\ve_1 = \vf_* \vv_1$, $\ve_2 = \vf_* \vv_2$ and $\ve_3 = \ve_1 \times \ve_2$.  The graphs of these lifts may be characterized as integral surfaces of an exterior differential system $\I_0$ which we will refer to as the {\em isometric immersion system}:

\begin{prop}[Prop. 1 in \cite{CITV}]  Let $\I_0$ be the Pfaffian system on $\Sigma_0:=\FM \times \FR$ generated by the 1-forms
$$\theta_0 := \w^3, \quad \theta_1 := \w^1 - \eta^1, \quad \theta_2 := \w^2 - \eta^2,\quad \theta_3 := \w^1_2 - \eta^1_2.$$
Then a surface $S\subset \Sigma_0$ is the graph of a lift of an isometric immersion defined on an open subset of $M$ if and only
if $S$ is an integral surface of $\I_0$ satisfying the independence condition $\eta^1 \wedge \eta^2 \ne 0$.
\end{prop}

Although $\I_0$ is defined on the 9-dimensional manifold $\Sigma_0$, it is the pullback of a Monge system of class 4 defined on an 8-dimensional quotient. To see this, note that there is a natural $SO(2)$-action on the fibers of $\FM$ that takes frame $(\vv_1, \vv_2)$ to a rotated
frame $(\cos \phi\, \vv_1 - \sin\phi \, \vv_2, \cos\phi \, \vv_2 + \sin\phi \, \vv_1)$.  When we similarly define an $SO(2)$-action on the fibers of $\FR$
that rotates vectors $\ve_1, \ve_2$ through an angle $\phi$, then the diagonal action of $SO(2)$ on $\Sigma_0$ gives a Cauchy characteristic
symmetry of $\I_0$.  Thus, $\I_0$ is the pullback of a well-defined EDS on the quotient manifold  $\Sigma_0/SO(2)$, and this EDS is a Monge system.  For the sake of convenience, however, we will calculate the structure of $\I_0$ on $\Sigma_0$.  For example, its generator 2-forms are
$$\Upsilon_1 = \w^3_1 \wedge \eta^1 + \w^3_2 \wedge \eta^2, \quad \Upsilon_2 = \w^3_2 \wedge \w^3_1 -K \eta^1 \wedge \eta^2.$$
It is easy now to check that the Monge system of elliptic, hyperbolic or parabolic type at a point of $\Sigma_0$ according to whether $K>0$, $K<0$ and $K=0$ at the underlying point on $M$.  Again assuming that the system is of constant type, we can characterize those metrics on $M$ for which
$\I_0$ is Darboux-integrable:

%quote k-result and normal forms from CITV
\begin{thm}[Theorem 2 in \cite{CITV}]\label{kcondthm}
Let $M$ be a Riemannian surface with either strictly positive or strictly negative curvature $K$; in the first case let $\varepsilon =1$, in the second case let $\varepsilon=-1$, and let $k=\sqrt{\varepsilon K}$.    Then the system $\I_0$ is Darboux integrable if and only if $k$ satisfies
\begin{equation}\label{kcond}
k_{11} = \dfrac52 \dfrac{k_1^2}{k}- 2\varepsilon k^3, \quad k_{12} = \dfrac52 \dfrac{k_1 k_2}{k}, \quad k_{22} = \dfrac52 \dfrac{k_2^2}{k}- 2\varepsilon k^3,
\end{equation}
where $k_a$ and $k_{ab}$ for $1\le a,b \le 2$ denote the components of the first and second covariant derivatives of $k$ with respect to an orthonormal
coframe on $M$.
\end{thm}

\begin{thm}[Theorem 5 in \cite{CITV}]\label{metriclist}
Let $M$ is a Riemannian surface with metric $g$ satisfying the conditions in Theorem \ref{kcondthm}.  For $K>0$, there are local coordinates $x,y$ on $M$ in which the metric is one of the following:
\begin{align*}
g_1 &= (\cosh^4 x) dx^2 + (\sinh^2 x) dy^2, \\
g_2 &= (\sinh^4 x) dx^2 + (\cosh^2 x) dy^2,\\
g_3 &= x^2 (dx^2 + dy^2)
\end{align*}
for $x>0$.  For $K<0$, there are local coordinates in which the metric is
$$g_4 = (\cos^4 x) dx^2 + (\sin^2 x) dy^2, \qquad x \in (0,\pi/2).$$
\end{thm}

% in order to clarify the relationship with Darboux equation, define prolongation of I0
As we will see below, this is exactly the set of metrics for which the Darboux equation \eqref{DarbouxPDE} is integrable by the method of Darboux.  In order to establish this,
we need to calculate the prolongation $\I_1$ of $\I_0$, which we will in turn connect to the Darboux equation in the next section.

If $\vE \in T_q \Sigma_0$ is an admissible integral 2-plane of $\I_0$, then the vanishing of $\Upsilon_1\restr_\vE$ shows that
there are numbers $h_{ab}=h_{ba}$ such that
\begin{equation}\label{anniE}
(\w^3_a -h_{ab} \eta^b)\restr_\vE = 0,\qquad 1\le a,b\le2,\end{equation}
and the vanishing of $\Upsilon_2 \restr_\vE$ shows that these numbers must satisfy
\begin{equation}\label{gausseq}
h_{11} h_{22} - h_{12}^2 = K.\end{equation}

If $S$ is an admissible integral surface of $\I_0$ then the $h_{ab}$ are functions along $S$ that give the components of the second fundamental
form $\II$ of $f(M)$ under the corresponding isometric immersion $f$, with respect the orthonormal frame $\ve_1, \ve_2$.  Then we see that \eqref{gausseq}
is nothing but the Gauss equation.  Moreover, the $h_{ab}$ uniquely determine the image of $\vE$ under the quotient map to $\Sigma_0/SO(2)$.
Thus, we may regard the $h_{ab}$ as fiber coordinates\footnote{Of course these are not really coordinates, since the Gauss equation makes them functionally dependent.  When necessary, we may introduce independent fiber coordinates $\lambda, \mu$ by setting
$$h_{11} = k e^\lambda \cosh \mu, \quad h_{12} = k \sinh \mu, \quad h_{22} = k e^{-\lambda} \cosh \mu$$
in the elliptic case, or a similar parametrization with $\cosh$ and $\sinh$ interchanged in the hyperbolic case.}
for the prolongation $\I_1$, and the generator 1-forms are the pullbacks of the 1-forms of $\I_0$ with the addition of the
1-forms on the left in \eqref{anniE} which annihilate $\vE$.

More formally, we let $h_{ab}$ be coordinates on the space $S^2(\R^2)$ of symmetric
$2\times 2$ matrices, let $\manH \subset S^2(\R^2)\times \FM$ be the hypersurface defined by the Gauss equation \eqref{gausseq}, let $\Sigma_1:=\manH \times \FR$, and let $\pmap$ denote the restriction to $\Sigma_1$ of the projection from $S^2(\R^2)\times \Sigma_0$ onto the second factor.
Then we define the prolongation as the Pfaffian system $\I_1$ on $\Sigma_1$ generated by the pullbacks $\pmap^*\theta_0, \ldots, \pmap^* \theta_3$ of the generators from above, and the new 1-forms
$$\nu_1 := \w^3_1 - h_{1a} \eta^a, \quad \nu_2 := \w^3_2 - h_{2a} \eta^a.$$
The 2-form generators of $\I_1$ are computed to be
\begin{equation}\label{codfish}
Dh_{1a} \wedge \eta^a \text{ and } Dh_{2a} \wedge \eta^a,\end{equation}
where
$$Dh_{ab} := dh_{ab} - h_{ac} \eta^c_b - h_{cb} \eta^c_a, \qquad 1\le a,b,c \le 2.$$
Because the vanishing of the 2-forms in \eqref{codfish} is easily seen to be equivalent to the condition that the covariant derivative $\nabla \II$ is a totally symmetric $(3,0)$-tensor, we see that $\I_1$ is equivalent to the Gauss-Codazzi system for the second fundamental form of the immersion.

In general, the prolongation of a Darboux-integrable system is also Darboux-integrable, but (as mentioned earlier) a system that is not Darboux-integrable may become Darboux-integrable after sufficiently many prolongations.  In the case of $\I_0$ we find that no additional metrics exist
for which the prolongation $\I_1$ is Darboux-integrable:

%state prolongation result
\begin{thm}\label{prongthm} The Gauss-Codazzi system $\I_1$ is Darboux-integrable only when the conditions \eqref{kcond} in Theorem \ref{kcondthm} hold.
\end{thm}

\begin{proof}
The proof is a lengthy calculation which unfortunately cannot be included here for reasons of space; however, we will briefly sketch the key idea.
%Recall that the derived system of a Pfaffian system $I$ on manifold $\Sigma$ is spanned by the sections of $I$ that are in the kernel of the map
%$\delta: \theta \mapsto d\theta \mod \I$.  Applying this construction iteratively results in the {\em derived flag} $I \supset I^{(1)} \supset I^{(2)} \ldots $.  As the derived flag terminates in the highest-rank Frobenius system contained in $I$, the Darboux integrability condition is equivalent to requiring that the derived flag of each characteristic system terminate in a Frobenius system of rank at least two.

Let $\V_+^{(k)}$ (respectively, $\V_-^{(k)}$) denote the members of the derived flag of the characteristic system $\V_+$ (resp., $\V_-$) of $\I_0$, and similarly let $\Vchat_{\pm}^{(k)}$ denote members of the derived flags of the characteristic systems of the prolongation.   In \cite{CITV}, it was shown that
$\V_+^{(1)}, \V_+^{(2)}$ have ranks 5 and 3 respectively for any metric where $K$ is not constant, and $\V_+^{(3)}$ has rank 2 if and only if the conditions \eqref{kcond} hold, in which
case $\V_+^{(3)}$ is a Frobenius system.  (All statements for the `plus' systems carry over without change to the `minus' systems.)  Similarly, we find
that $\Vchat_+^{(1)}, \Vchat_+^{(2)}, \Vchat_+^{(3)}$ generically have ranks 7, 6 and 4 respectively.
It is standard that $\pmap^* \V_+ \subset \Vchat_+$, and consequently $\pmap^* \V^{(k)}_+ \subset \Vchat_+^{(k)}$, but in this case we also find
that $\pmap^* \V_+ \subset \Vchat_+^{(1)}$, and hence $\pmap^* \V^{(k)}_+ \subset \Vchat_+^{(k+1)}$.  In particular, $\Vchat_+^{(3)}$ contains
the rank 3 subsystem $\pmap^*\V_+^{(2)}$.

Suppose that the conditions \eqref{kcond} do not hold, but nevertheless $\Vchat_+^{(3)}$ contains a Frobenius subsystem $\sysF$ of rank 2.  Thus, $\Vchat_+^{(4)}$ must have rank at least two, and thus $\V_+^{(3)}$ must have rank at least one.  Because the conditions in \eqref{kcond} do not hold, these ranks must be exactly two and one respectively.  Suppose $\V_+^{(3)}$  is spanned by a 1-form $\alpha$; then $\pmap^*\alpha$ must belong to $\sysF = \Vchat_+^{(4)}$.  By writing $\pmap^*\alpha$ as a linear combination of two exact differentials, one sees that $\alpha$ must have Pfaff rank at most one, i.e.,  $\alpha\wedge d\alpha$ may be nonzero but $\alpha \wedge (d\alpha)^2 =0$.   However, we calculate that whenever $\V_+^{(3)}$ has rank one, its generator satisfies $\alpha \wedge (d\alpha)^2 \ne 0$.
\end{proof}

\section{Integrability of the Darboux Equation}\label{DEsec}
In this section we will set up the Darboux equation as a Monge-Amp\`ere EDS $\J_0$, show how it is related to the Gauss-Codazzi system $\I_1$, and determine the metrics for which the prolongation $\J_1$ of $\J_0$ is Darboux-integrable.

% do as EDS (two ways?)
Given a function $u$ on $M$, the components of the first and second covariant derivatives of $u$ with respect to an orthonormal coframe depend on
the choice of coframe, and are thus well-defined as functions on $\FM$ rather than $M$.  In terms of the canonical and connection forms on $\FM$ these components are defined
\begin{equation}\label{ucovs}
du = u_a \eta^a, \qquad du_a = u_b \eta^b_a + u_{ab} \eta^b,
\end{equation}
where, in the left-hand side of the first equation, $u$ is pulled back to $\FM$ via the projection to $M$.
%where $\etat^1, \etat^2, \etat^1_2$ denote the pullbacks of $\eta^1, \eta^2, \eta^1_2$ under the section $f:M \to \FM$ that gives the frame.
In terms of these components, the Darboux equation \eqref{DarbouxPDE} is written as
\begin{equation}\label{DEcomponentwise}
u_{11} u_{22} - u_{12}^2 =  (1-u_1^2 - u_2^2)K.
\end{equation}
The combinations of components on both sides are invariant under pointwise rotations of the frame, so this gives a well-defined condition on $M$.

As stated above, we wish to define a Monge-Amp\'ere system that is equivalent to the partial differential equation \eqref{DEcomponentwise}.
We note that, in light of \eqref{ucovs}, the equation \eqref{DEcomponentwise} is equivalent to the vanishing of the 2-form
\begin{equation}\label{DEformdef}
\Upsilon:=(du_1 - u_2 \eta^2_1) \wedge (du_2 -u_1 \eta^1_2) - K(1-u_1^2 - u_2^2) \eta^1 \wedge \eta^2.
\end{equation}
Based on this observation, we will define $\J_0$ by beginning with $\FM$ and adding $u, u_1, u_2$ as extra variables which are constrained by the nondegeneracy condition $|du|^2 < 1$.  In other words,
$\J_0$ is defined on the 6-dimensional manifold $\FM \times \manW$, where $\manW $ denotes the open subset of $\R^3$ where the coordinates $u,u_1,u_2$ satisfy $u_1^2 + u_2^2<1$. The algebraic generators of $\J_0$ will be the 1-form
\begin{equation}\label{psi0formdef}
\psi_0:=du - u_1 \eta^1 - u^2 \eta^2,
\end{equation}
and the 2-forms $d\psi_0$ and $\Upsilon$.  This system is elliptic at points where $K>0$ and hyperbolic at points where $K<0$.  Moreover, as with systems $\I_0$ and $\I_1$ above, $\J_0$ has a Cauchy characteristic symmetry corresponding to pointwise rotations of the frame, and thus
descends to be a well-defined Monge-Amp\`ere system on the five-dimensional quotient of $\FM \times \manW$ under the diagonal action of $SO(2)$.
(Alternatively, we may break the symmetry by specializing to a particular choice of frame, which will do when we work with specific metrics below.)

We are able to define an EDS encoding the Darboux equation on a relatively low-dimensional space, using only $u$ and its first derivative components as variables,  because the equation is only mildly nonlinear in the second derivatives of $u$.  However, to make the connection with the Gauss-Codazzi system we must pass to the prolongation $\J_1$, and thus it is
 necessary to introduce the second covariant derivative components as new variables, and define additional 1-forms
whose vanishing is equivalent to the equations \eqref{ucovs}.   Accordingly, on $\FM \times \R^6$ (where we take $u$, $u_a$ and $u_{ab}=u_{ba}$ as coordinates on the second factor) define the 1-forms
\begin{equation}\label{psiformsdef}
\psi_0 := du - u_a \eta^a, \quad \psi_a:=du_a - u_b \eta^b_a - u_{ab}\eta^b, \qquad 1\le a,b \le 2,
\end{equation}
let $\Sigma_2 \subset \FM \times \R^6$ be the 8-dimensional submanifold where \eqref{DEcomponentwise} holds and $u_1^2 + u_2^2 < 1$
(note that this condition ensures that $\Sigma_2$ is smooth), and let $\J_1$ be the Pfaffian system on $\Sigma_2$ generated by $\psi_0, \psi_1, \psi_2$ and their exterior derivatives.  Then $\J_1$ is a Monge system of class 3 and is the prolongation of $\J_0$.

We now define a mapping $\rmap$ that takes integral surfaces of $\I_1$ to integral surfaces of $\J_1$.  (This mapping will merely
formalize the derivation of the Darboux equation from the Gauss-Codazzi system that we gave earlier in \S\ref{introsec}.)
As before, we fix a unit vector $\hatj$ in $\R^3$, and let $\tilde\Sigma_1$ be the open subset of $\Sigma_1$ where $\ve_3\cdot \hatj \ne 0$.
We will first define $\rmap:\tilde \Sigma_1 \to \FM \times \R^6$ by requiring $\rmap$ to be the identity on the $\FM$-factor, and letting
\begin{equation}\label{rdefuvals}
u = \vx \cdot \hatj, \qquad u_a = \ve_a \cdot \hatj, \qquad u_{ab} = h_{ab} (\ve_3 \cdot \hatj).
\end{equation}
Then, because the $h_{ab}$ satisfy the Gauss equation \eqref{gausseq}, and $(u_1)^2 + (u_2)^2 = 1-(\ve_3\cdot \hatj)^2 < 1$, the image of $\rmap$ lies in $\Sigma_2$.

% show integrable extension
\begin{prop} The mapping $\rmap:\tilde\Sigma_1 \to \Sigma_2$ makes $\I_1$ an integrable extension of $\J_1$.
\end{prop}
\begin{proof} We must show first that the pullbacks of the generator 1-forms of $\J_1$ lie in $\I_1$, and then define a complementary sub-bundle $E$
such that $\I_1$ is generated algebraically by sections of $E$ and the pullbacks of forms in $\J_1$.

First, using \eqref{d-of-basepoint}, \eqref{d-of-frame} and \eqref{rdefuvals} we compute
$$\rmap^*du = (\ve_a \cdot \hatj) \w^a, \quad \rmap^*du_a = (\ve_i \cdot \hatj) \w^i_a.$$
From these it follows, respectively, that
$$\rmap^* \psi_0 = (\ve_a \cdot \hatj) \theta_a,$$
and
\begin{equation}\label{drpsi}
\begin{aligned}
\rmap^*\psi_1 &= - (\ve_2 \cdot \hatj) \theta_3 + (\ve_3 \cdot \hatj) \nu_1, \\
\rmap^*\psi_2 &= (\ve_1 \cdot \hatj) \theta_3 + (\ve_3 \cdot \hatj) \nu_2.
\end{aligned}
\end{equation}
Thus, the pullbacks of the generator 1-forms of $\J_1$ span a rank-3 sub-bundle within $I_1$.  We define $E$ as spanned by $\theta_0, \theta_3$ and $(\ve_1 \cdot \hatj) \theta_2 - (\ve_2 \cdot \hatj) \theta_1$.  Thus,
the 1-forms of $\I_1$ may be expressed as linear combinations of sections of $E$ and $\rmap^*\psi_0, \rmap^*\psi_1, \rmap^*\psi_2$.
Since $\I_1$ is a prolongation, its first derived system is spanned by the pullbacks from $\Sigma_1$ of the 1-forms $\theta_0, \ldots, \theta_3$ of $\I_0$,
and thus its 2-form generators are $d\nu_1$ and $d\nu_2$.
Because the rightmost coefficients in \eqref{drpsi} are nonzero, $\rmap^* d\psi_1$ and $\rmap^* d\psi_2$ give an equivalent pair
of 2-form generators.  Thus, $\I_1$ is algebraically generated by $\rmap^*\J_1$ and sections of $E$.
\end{proof}

% quote AF result
\begin{cor}\label{jcor}  If $\J_1$ is Darboux-integrable, then $\I_1$ is Darboux-integrable.
\end{cor}
\begin{proof} Because $\I_1$ is an integrable extension of $\J_1$, this follows by Theorem 5.1 in \cite{AFB}.
\end{proof}

This lets us prove our main result:
\begin{thm}  The Darboux equation \eqref{DarbouxPDE} associated to a metric $g$ on surface $M$ is Darboux-integrable at the 2-jet level if and only if $g$ is locally isometric to one of the metrics listed in Theorem \ref{metriclist}.
\end{thm}
\begin{proof}  From  Corollary \ref{jcor} it follows that if the system $\J_1$ is Darboux-integrable then $\I_1$ is Darboux-integrable, and hence by
Theorem \ref{prongthm} the metric must satisfy the conditions \eqref{kcond}.  It only remains to check that, for each of the metrics listed in Theorem
\ref{metriclist}, the system $\J_1$ is in fact Darboux-integrable.

We begin with metric $g_4$, since in this case the $\J_0, \J_1$ are hyperbolic and the characteristic systems are real.  Before showing
that $\J_1$ is Darboux-integrable, we will calculate generators and characteristic systems for $\J_0$ on $M\times \manW$, pulling the
previously-defined system back by a specific choice of orthonormal coframe on $M$:
$$\etat^1 = \cos^2 x\,dx, \quad \etat^2 = \sin x\,dy,$$
Then the structure equations \eqref{Mstruc} imply that $\etat^1_2 = -\sec x\,dy$ and $K=-\sec^4 x$.
In place of the coordinates $u_1, u_2$ on $\manW$ we will use alternate coordinates $\theta, \phi$ related to them by
\begin{equation}\label{u1u2change}
u_1 = \sin\phi \cos\theta, \quad u_2=\sin\phi \sin \theta,\qquad \phi \in (0,\pi/2).
\end{equation}
Using \eqref{DEformdef} and \eqref{psi0formdef} we now express the generators of $\J_0$ as
\begin{equation}\label{psi0form}
\psi_0 = du - \sin\phi\, (\cos \theta\, \etat^1 + \sin \theta\, \etat^2)
\end{equation}
and
$$\Upsilon = \cos\phi\sin\phi\,d\phi \wedge (d\theta + \sec x\, dy) + \cos^2 \phi \sin x \sec^2 x \, dx \wedge dy.$$
By solving for the null directions of the quadratic form $Q$ defined in \S2 above, we find that the decomposable 2-form
generators of the ideal are $\Upsilon \mp (\cos\phi \sec^2 x) d\psi_0$, and the characteristic systems are
$$V_\pm = \{\psi_0, d\phi \pm (\sec^2x)\sigma, d\theta +\sec x\,dy \pm (\cot\phi \sec^2x )\tau \}$$
where, for the sake of convenience, we let
\begin{equation}\label{sigmataudef}
\sigma:=(\sin\theta)\etat^1 - (\cos\theta)\etat^2, \qquad \tau:=(\cos\theta)\etat^1 +(\sin\theta)\etat^2.
\end{equation}

On any admissible integral surface of $\J_0$, each system $V_+$, $V_-$ pulls back to have rank one.  In particular, the last two members of each system restrict to be linearly dependent. Thus, the prolongation
$\J_1$ is generated by $\psi_0$ and
\begin{align*}
\tpsi_1 &:= \tan\phi(d\theta +\sec x\,dy) +(\sec^2x )\tau -p(d\phi +(\sec^2x)\sigma),\\
\tpsi_2 &:= \tan\phi(d\theta +\sec x\,dy) -(\sec^2x)\tau -q(d\phi -(\sec^2x)\sigma),
\end{align*}
where we introduce $p,q$ as new variables that give the coefficients of these linear dependencies.
After subtracting off suitable multiples of $\tpsi_1, \tpsi_2$, the exterior derivatives $d\tpsi_1$ and $d\tpsi_2$ become
decomposable 2-forms.   Each characteristic system of $\J_1$ is generated by the 1-forms of $\J_1$ and the factors of one
of these decomposable 2-forms; for example, the factors in $d\tpsi_1$ lead to
%The exterior derivative of each of these forms is we characteristic systems of $\J_1$ yields
$$\Vhat_+ =\{ \psi_0, \tpsi_1, \tpsi_2, \tau-q\sigma, \cot x (\cos^2 x\,dp-\cot\phi (1+p^2)\tau) +(\sin \theta+ p\cos\theta) (\tau -p\sigma) \},
$$
while the generators of the other characteristic system are obtained by switching $p$ and $q$, and changing angle $\phi$ to its complement $\pi/2 -\phi$.

Each characteristic system turns out to contain a rank 2 Frobenius system, so that $\J_1$ is Darboux-integrable.  In the case of $\Vhat_+$, the Frobenius system is
\begin{multline*}
\Vhat_+^{(\infty)} = \{ \cos \theta \,d\theta + (\sin\theta \cot\phi + \cot x)d\phi + (\sin \theta\cot x + \cot \phi)dx, \\
dp+p\tan x (\sin\theta+p \cos\theta)d\phi + \left( (p+\tan\theta)\tan x-(1+p^2)\sec\theta\cot\phi\right)(dx +\sin\theta\,d\phi)
\}.
\end{multline*}
We note that first 1-form does not involve the prolongation variables $p,q$, and thus is the pullback of a 1-form defined on $M \times \manW$; in fact, this 1-form is integrable, and spans the terminal derived system of $V_+$.
The Frobenius system inside $\Vhat_-$ is obtained from this one by the same switching of variables mentioned above, and also contains
the pullback of the terminal derived system of $V_-$.

We now turn to the metrics $g_1, g_2, g_3$, for which the Gauss curvature is positive, and hence the systems $\J_0, \J_1$ are of elliptic type.
The decomposable 2-forms in this case will be complex-valued, and the same goes for the 1-forms generating the characteristic systems, but for $\J_0$ and $\J_1$ each pair of characteristic systems will be related by complex conjugation, so that it is only necessary to do calculations for one of the pair.

Other than the complexification, the calculation for metric $g_1$ resembles that for $g_4$, except that $\cos x$ and $\sin x$ are replaced by hyperbolic functions. The orthonormal
coframe on $M$ is
$$\etat^1 = \cosh^2 x\,dx, \quad \etat^2 = \sinh x\,dy,$$
resulting in $\etat^1_2 = -\sech x\,dy$ and $K=\sech^4 x$.  We make the same change of variable as in \eqref{u1u2change}, so that
the 1-form generator of $\J_0$ takes the same form as in \eqref{psi0form}.  The 2-form generator of $\J_0$ is
$$\Upsilon = \cos\phi\sin\phi\,d\phi \wedge (d\theta + \sech x\, dy) + \cos^2 \phi \sinh x \sech^2 x \, dx \wedge dy.$$
The decomposable 2-forms in the ideal are $\Upsilon \mp \ri (\cos\phi \sech^2 x) d\psi_0$; taking the factors in the case of the upper sign
gives the characteristic system
$$V_+ = \{ \psi_0, d\phi + \ri (\sech^2 x )\sigma, d\theta + \sech x\, dy + \ri (\cot\phi \sech^2 x) \tau \}$$
where we again take $\sigma, \tau$ as in \eqref{sigmataudef}.

The prolongation $\J_1$ is generated by $\psi_0$ and the real and imaginary parts of
$$\tpsi_1 := \tan \phi ( d\theta +\sech x\,dy) +\ri (\sech^2x ) \tau -(p+\ri q) (d\phi +\ri (\sech^2x) \sigma) \}$$
where new variables $p,q$ are real-valued.  Computing the exterior derivative $d\tpsi_1$ modulo the 1-form generators of $\J_1$ yields a decomposable 2-form
whose factors allow us to form the characteristic system
$$\Vhat_+ = \{ \psi_0, \Re\tpsi_1, \Im\tpsi_1, \tau - \overline{z} \sigma,  \coth x( \cosh^2 x\, dz -\ri \cot\phi(1+z^2) \tau) + (\sin \theta + z\cos \theta)(z\sigma-\tau)\},$$
where we have set $z=p+\ri q$.  (Again, the other characteristic system is the complex conjugate of this one.) Each characteristic system contains a rank 2 Frobenius system, for example
\begin{multline*}
\Vhat_+^{(\infty)} =  \{ \cos \theta \,d\theta + (\sin\theta \cot\phi -\ri \coth x)d\phi + (\sin \theta\coth x +\ri \cot \phi)dx, \\
dz +\ri z \tanh x (\sin \theta+ z\cos\theta) d\phi - \left( (z+\tan\theta) \tanh x+ \ri (1+z^2)\sec\theta \cot\phi\right)(dx - \ri \sin\theta d\phi)\}.
\end{multline*}
Again, the first 1-form in this system is integrable, and is the pullback of a 1-form spanning the terminal derived system of $V_+$.

The calculation for metric $g_2$ is similar, but with $\cosh x$ and $\sinh x$ interchanged.

The calculation for metric $g_3$ is similar to $g_1$ and $g_2$ but simpler in form.  The orthonormal coframe on $M$ is
$$\etat^1 = x\,dx, \quad \eta^2 = x\,dy,$$
resulting in $\eta^1_2 = -(1/x) dy$ and $K = 1/x^4$. The variable change \eqref{u1u2change} and form of $\psi_0$ are the same. The 2-form generator of
$\J_0$ is
$$\Upsilon =\cos\phi\sin\phi\,d\phi \wedge (d\theta +x^{-1} dy) + x^{-2} \cos^2 \phi \, dx \wedge dy.$$
We define $\sigma, \tau$ as in \eqref{sigmataudef}, and in terms of these one characteristic system is
$$V_+ = \{\psi_0, x^2 d\phi + \ri \sigma, x^2 d\theta + x\, dy + \ri (\cot\phi)\tau\}.$$
The prolongation $\J_1$ is again generated by $\psi_0$ and the real and imaginary parts of
$$\tpsi_1 := \tan \phi ( x^2 d\theta + x\, dy) + \ri\tau -(p+\ri q)  (x^2 d\phi + \ri \sigma).$$
The corresponding characteristic system of the prolongation is
$$\Vhat_+ =  \{ \psi_0, \Re\tpsi_1, \Im\tpsi_1, \tau - \overline{z} \sigma,
x^2 dz -\ri \cot\phi (1+z^2) \tau + (\sin\theta+ z \cos\theta)(z\sigma-\tau) \},$$
and this contains a rank 2 Frobenius system
\begin{multline*}
\Vhat_+^{(\infty)} = \{ \cos\theta d\theta + (\sin\theta \cot\phi - \ri) d\phi + (\sin\theta + \ri \cot\phi) x^{-1} dx,\\
dz +\ri z (\sin\theta+ z\cos\theta)d\phi - (z + \tan \theta+\ri (1+z^2)\sec\theta \cot\phi) (x^{-1}dx -\ri \sin\theta d\phi).
\end{multline*}
\end{proof}

\section*{Acknowledgements}
The author is grateful to the organizers of a special session on geometry of submanifolds, held in honor of the 70th birthday of Professor Bang-Yen Chen, for inviting this submission.
The author also acknowledges the support of the National Science Foundation, through an Independent Research and Development plan undertaken by the author while an NSF employee.

\end{document}